\documentclass[11pt,a4paper,reqno]{amsart}

\usepackage{aliascnt}
\usepackage{hyperref}
\usepackage{amsfonts}
\usepackage{mathrsfs}
\usepackage{amsmath}
\usepackage{amsmath,graphics}
\usepackage{amssymb}
\usepackage{indentfirst,latexsym,bm,amsmath,pstricks,amssymb,amsthm,graphicx,fancyhdr,float,color}
\usepackage[all]{xy}
\usepackage{amsthm}
\usepackage{amscd}
\usepackage[all]{hypcap}

\newtheorem{theorem}{Theorem}[section]

\newaliascnt{lemma}{theorem}
\newtheorem{lemma}[lemma]{Lemma}
\aliascntresetthe{lemma}

\newaliascnt{conjecture}{theorem}

\aliascntresetthe{conjecture}

\newaliascnt{proposition}{theorem}
\newtheorem{proposition}[proposition]{Proposition}
\aliascntresetthe{proposition}

\newaliascnt{corollary}{theorem}
\newtheorem{corollary}[corollary]{Corollary}
\aliascntresetthe{corollary}

\newaliascnt{problem}{theorem}

\aliascntresetthe{problem}

\newaliascnt{question}{theorem}
\newtheorem{question}[question]{Question}
\aliascntresetthe{question}

\newaliascnt{claim}{theorem}
\newtheorem{claim}[claim]{Claim}
\aliascntresetthe{claim}

\theoremstyle{definition}

\newaliascnt{definition}{theorem}
\newtheorem{definition}[definition]{Definition}
\aliascntresetthe{definition}

\newaliascnt{example}{theorem}
\newtheorem{example}[example]{Example}
\aliascntresetthe{example}

\theoremstyle{remark}

\newaliascnt{remark}{theorem}
\newtheorem{remark}[remark]{Remark}
\aliascntresetthe{remark}

\newaliascnt{remarks}{theorem}

\aliascntresetthe{remarks}

\numberwithin{equation}{section}

\numberwithin{figure}{section}

\def\wt{\widetilde}

\def\({$($}
\def\){$)$}

\def\bbp{\mathbb P}

\def\Pic{\text{{\rm Pic\,}}}

\def\Alb{{\rm Alb}}

\def\Alb{\mathrm{Alb}}

\begin{document}

\title{Upper bounds on the genus of Hyperelliptic  Albanese fibrations}	
\author{Songbo Ling}

\address{School of Mathematics, Shandong University, Jinan 250100, People's Republic of China}

\email{lingsongbo@sdu.edu.cn}

\author{Xin L\"u}

\address{School of Mathematical Sciences,  Key Laboratory of MEA(Ministry of Education) \& Shanghai Key Laboratory of PMMP,  East China Normal University, Shanghai 200241, China}

\email{xlv@math.ecnu.edu.cn}

\thanks{This work is supported by National Natural Science Foundation of China,  Shanghai Pilot Program for Basic Research (No. TQ20240202), Fundamental Research Funds for central Universities, Science and Technology Commission of Shanghai Municipality (No. 22DZ2229014) and  Natural Science Foundation of Shandong Province (No.  ZR2023QA00).}


%
%




	\maketitle

	\begin{abstract}
		Let $S$ be a minimal irregular surface of general type, whose Albanese map induces a hyperelliptic fibration $f:\,S \to B$ of genus $g$.
		We prove a quadratic upper bound on the genus $g$, i.e., $g\leq h\big(\chi(\mathcal{O}_S)\big)$, where $h$ is a quadratic function.
		We also construct examples showing that the quadratic upper bounds can not be improved to linear ones.
		In the special case when $p_g(S)=q(S)=1$, we show that $g\leq 14$.

%
	\end{abstract}

	\section{Introduction}
	We work over the complex number throughout this paper.
	Let $S$ be a minimal irregular surface of general type, and $a:\,S \to \Alb(S)$ be  its Albanese map.
	We are interested in the case when the image $a(S)$ is a curve.
	In this case, the Albanese map induces a fibration, which we call the Albanese fibration of $S$:
	$$f:\,S \longrightarrow B.$$
	In fact, by the universal property of the Albanese map, $B \cong a(S)$, and under this isomorphism the above fibration $f$ is nothing but the Albanese map of $S$.
	
	Let $g$ be the genus of a general fiber of $f$.
	A natural problem is to study the behavior of the genus $g$.
	\begin{question}
		Can we give an  upper bound on the genus $g$ of the Albanese fibration of $S$?
	\end{question}
	By \cite{Cat00},
	it is known that the genus $g$ of the Albanese fibration of $S$ is a differential invariant, and hence is also a deformation invariant.
	Fixing $\chi(\mathcal{O}_S)$ and $K_S^2$, there are only finitely many deformation equivalence classes of such surfaces, cf. \cite[\S\,VII]{bhpv}.
	Hence there must be an upper bound on $g$ depending on $\chi(\mathcal{O}_S)$ or $K_S^2$.
	However, it is still unclear how the upper bound depends on $\chi(\mathcal{O}_S)$ (or $K_S^2$).
	Explicit upper bounds are only known under some extra assumptions.

	\begin{enumerate}
		\item Suppose that $g(B)>1$. Then $g\leq \frac{\chi(\mathcal{O}_S)}{g(B)-1}+1\leq \chi({\mathcal{O}_S})+1$
		by the semi-positivity of the Hodge bundle $f_*\mathcal{O}_S\big(K_{S/B}\big)$.
		Moreover, this bound is sharp since there are generalised hyperelliptic surfaces, whose Albanese
		map is a fibration of genus $g= \chi({\mathcal{O}_S})+1$ (cf. \cite{Cat00}).\vspace{2mm}
		
		\item Suppose that $g(B)=1$ and $K_S^2< 4\chi(\mathcal{O}_S)$. From the slope inequality \cite{CH88,Xiao87} it follows that $\frac{K_S^2}{\chi(\mathcal{O}_S)}\geq \frac{4(g-1)}{g}$, and hence
		$g\leq \frac{4\chi(\mathcal{O}_S)}{4\chi(\mathcal{O}_S)-K_S^2}\leq 4\chi(\mathcal{O}_S)$.
		Konno \cite{konno96} showed that $g\leq 6$ if moreover $S$ is an even canonical surface.\vspace{2mm}
		
		\item Suppose that $g(B)=1$ and $K_S^2 = 4\chi(\mathcal{O}_S)$.
		We showed that $g\leq \max\{6, 3\chi(\mathcal{O}_S)+1\}$, cf. \cite{ll25}.
See also \cite{Ish05} for the case when $f$ is moreover hyperelliptic, where it was proved that $g\leq 2\chi(\mathcal{O}_S)+2$.
	\end{enumerate}
	
	All the upper bounds above are linear in $\chi(\mathcal{O}_S)$.
	Our first aim is to give an upper bound on the genus $g$ of hyperelliptic Albanese fibrations of surfaces of general type with $K_S^2 > 4\chi(\mathcal{O}_S)$.
	\begin{theorem}\label{thm-upper1}
		Let $S$ be a minimal irregular surface of general type with   $q(S)=1$ and $K_S^2 > 4\chi(\mathcal{O}_S)$, and $f:\,S \to B$ its Albanese fibration whose general fiber is of genus $g$.
		Suppose that $f$ is hyperelliptic. Then
	\begin{equation}\label{eqn-upper-1}
g\leq \frac{25}{4}{\chi(\mathcal{O}_S)}^2+\frac{19}{2}\chi(\mathcal{O}_S)+2.
\end{equation}
	\end{theorem}
	We will  construct in \autoref{sec-example} examples of minimal irregular surfaces of general type with $K_S^2 > 4\chi(\mathcal{O}_S)$ admitting hyperelliptic Albanese fibrations, whose general fiber is of genus $g$ as large as a quadratic function in $\chi(\mathcal{O}_S)$.
	This is different from the phenomenons when $K_S^2\leq 4\chi(\mathcal{O}_S)$, where the upper bounds are all linear in $\chi(\mathcal{O}_S)$ as mentioned above.
		
	\begin{remark}
		As mentioned above, the genus $g$ is a deformation invariant when $f$ is the Albanese fibration.
		Hence one may expect an upper bound on $g$ using invariants of $S$.
		In general, there is no such upper bound on $g$ if $f$ is not the Albanese fibration.
		In fact, Penegini-Polizzi (\cite{PP17}) constructed a minimal surface $S$ with $p_g(S)=q(S)=2,K_S^2=4$ (whose Albanese map is generically finite), on which there are fibrations $f_k:\,S \to B$, such that the genera of $f_k$'s can be arbitrarily large.

Nevertheless, in the case when $f:\,S \to B$ is hyperelliptic, the quadratic upper bound proved in \autoref{thm-upper1} holds true even if $f$ is not the Albanese fibration,
		and hence our results can be applied in a more general situation.
		See \autoref{theorem upper bound for g}
		and \autoref{thm-sharp-upper-bound} for more details.
	\end{remark}
	
	Another interest of ours is the classification of minimal irregular surface of general type with $\chi(\mathcal{O}_S)=1$.
	Although the value $\chi(\mathcal{O}_S)=1$ is the minimal possibility among surfaces of general type,
	the classification of such surfaces is widely open.
	Beauville proved in an appendix to \cite{deb82} that  $p_g(S)=q(S)\leq 4$ for such surfaces,
	and the equality holds if and only if $S$ is isomorphic to a product of two curves of genus $2$.
	
	We are more interested in the case when the Albanese map of $S$ induces a fibration.
	In this case, $p_g(S)=q(S)\leq 3$ by Beauville's result above.
	In a series of works \cite{ccm98,hp02,penegini11,pirola02,zucconi03}, a full classification has been obtained when $p_g(S)=q(S)= 3$ or $p_g(S)=q(S)=2$.
	However, the case when $p_g(S)=q(S)=1$ seems to be much more mysterious.
	An explicit upper bound on the genus $g$ of the Albanese fibration is still unknown,
	although theoretically such an upper bound exists since there are at most finitely many deformation equivalence classes of such surfaces.
	By \autoref{thm-upper1} above, $g\leq 17$ if the Albanese fibration is hyperelliptic.
In fact, we can even get more.
	\begin{theorem}\label{thm-upper-chi=1}
		Let $S$ be a minimal  surface of general type with $p_g(S)=q(S)=1$.
		Suppose that the Albanese fibration $f:\,S\to B$ of $S$ is hyperelliptic of genus $g$. Then $g\leq 14$. More precisely, it holds $g\leq 10$ except for the following two possible cases:
		
        \qquad \qquad {\rm (i)} $K^2=8,~g=14$; \qquad \qquad {\rm (ii)} $K^2=8,~g=11$.
	\end{theorem}
	
	The known examples of surfaces with $p_g(S)=q(S)=1$ are mostly restricted to the region $g\leq K_S^2$.
	The first example with $g>K_S^2$ was constructed in \cite{fp15}, in which $g=7$ and $K_S^2=6$.
	More recently, an example with $g=19$ and $K_S^2=9$ was constructed in \cite{cky17}.
	By our result above, the Albanese fibration of the example in \cite{cky17} must be non-hyperelliptic.

	\vspace{2ex}
The paper is organized as follows.
In \autoref{sec-preliminaries}, we mainly review some basic facts about surface fibrations and do some technical preparations.
In  \autoref{sec-upper-hyperelliptic} and \autoref{sec-pg=q=1}, we prove
\autoref{thm-upper1}  and \autoref{thm-upper-chi=1} respectively.
Finally in \autoref{sec-example}, we construct examples of irregular surfaces of general type
with hyperelliptic fibrations of genus $g$ as large as a quadratic function in $\chi(\mathcal{O}_S)$.


\vspace{2ex}	
{\bf Acknowledgments.}
We are grateful to the referees for valuable comments and suggestions,  which make our paper more clear.

\section{Preliminaries} \label{sec-preliminaries}
In this section, we mainly review some basic facts and fix the notations.
In \autoref{sec-fibration}, we recall some general facts about  surface fibrations, and refer to \cite{bhpv} for more details;
and then in \autoref{sec-hyperelliptic} we restrict ourselves to the theory on the hyperelliptic fibrations,
which goes back to Horikawa and Xiao, cf. \cite{hor77,Xiao91,Xiao92}.
\subsection{The surface fibrations}\label{sec-fibration}
Let $f:S\rightarrow B$ be a fibration of curves of genus $g\geq 2$, i.e. $f$ is a proper morphism from a smooth projective surface $S$ onto a smooth projective curve $B$ with connected fibers, and the general fiber is a smooth projective curve of genus $g$. The fibration $f$ is called {\em relatively minimal} if there is no $(-1)$-curve (i.e. a smooth rational curve with self-intersection $-1$) contained in the fibers of $f$. It is called {\em hyperelliptic} if its general fiber   is hyperelliptic, {\em smooth} if all its fibers are smooth, {\em isotrivial} if all its smooth fibers are mutually isomorphic, and {\em locally trivial} if it is both smooth and isotrivial.
	
	Let $\omega_S$ (resp. $K_S$) be the canonical sheaf (resp. canonical divisor) of $S$, and let $\omega_{S/B}=\omega_S\otimes f^*\omega_B^\vee$ (resp. $K_f=K_{S/B}=K_S-f^*K_B$) be the relative canonical sheaf (resp. the relative canonical divisor) of $f$. Put $b:=g(B)$, $p_g:=h^0(S,\omega_S)$, $q:=h^1(S,\omega_S)$, $\chi=\chi(\mathcal{O}_S):=p_g-q+1$, and let $e(S)$ be the topological Euler characteristic of $S$. The basic invariants of
	$f$ are:
	$$\begin{aligned}
	\chi_f&\,=\chi-(g-1)(b-1);\\
	K_f^2&\,=K_S^2-8(g-1)(b-1);\\
	e_f&\,=e(S)-4(g-1)(b-1).
	\end{aligned}$$
	These invariants satisfy the following properties:
	\begin{enumerate}
		\item $\chi_f=\deg f_*\omega_{S/B}$ is the degree of the Hodge bundle $f_*\omega_{S/B}$. Moreover $\chi_f\geq 0$, and the equality holds if and only if $f$ is locally trivial.
		\item When $f$ is relatively minimal, $K_f^2\geq 0$, and the equality holds if and only if $f$ is locally trivial.
		\item $e_f=\sum e_F$, where  $e_F:=e(F_{red})-(2-2g)$ for any fiber $F$,  $F_{red}$ is the reduced part of $F$ and $e(F_{red})$ is the topological Euler characteristic of $F_{red}$. Moreover, $e_F\geq 0$, and the equality holds if and only if $F$ is smooth. Hence $e_f\geq 0$, and $e_f=0$ if and only if  $f$ is smooth.
		\item The above three invariants satisfy the Noether equality: $12\chi_f=K_f^2+e_f$.
	\end{enumerate}
	
	Suppose that $f$ is relatively minimal and not locally trivial. Then both $K_f^2$ and $\chi_f$ are strictly positive. In this case, the {\em slope} of $f$ is defined to be $\lambda_f=\frac{K_f^2}{\chi_f}$.
	According to the non-negativity of these basic invariants of $f$,
	it holds $0<\lambda_f\leq 12$.
	The so-called slope inequality, proved independently by Cornalba-Harris \cite{CH88} and Xiao \cite{Xiao87}, states that
	\begin{equation}\label{eqn-slope}
		\lambda_f\geq \frac{4(g-1)}{g}.
	\end{equation}
	
	\subsection{Invariants of  hyperelliptic fibrations}\label{sec-hyperelliptic}
	In this subsection, we restrict to hyperelliptic fibrations.
	Let $f:S\rightarrow B$ be a relatively minimal hyperelliptic fibration. The relative canonical map of $f$ is generically of degree 2. Hence it determines an involution $\sigma$ on $S$ whose restriction on a general fiber $F$ of $f$ is the hyperelliptic involution of $F$.
	
	Let $\mu:\tilde{S}\rightarrow S$ be the blow-ups  of all the isolated fixed points of $\sigma$ and let $\tilde{\sigma}$ be the induced involution on $\tilde{S}$. The quotient space $\tilde{P}=\tilde{S}/\langle\tilde{\sigma}\rangle$ is a smooth surface, and $f$ induces a ruling $\tilde{h}:\tilde{P}\rightarrow B$ on $\tilde{P}$. Also, the quotient map $\tilde{\pi}: \tilde{S}\rightarrow \tilde{P}$ is a double cover which is determined by the pair $(\tilde{R},\tilde{\delta})$, where $\tilde{R}$ is the branch locus of $\tilde{\pi}$ and $\tilde{\delta}$ is a line bundle such that $\mathcal{O}_{\tilde{P}}(\tilde{R})\cong \tilde{\delta}^{\otimes 2}$ and $\tilde{\pi}_*\mathcal{O}_{\tilde{S}}\cong \mathcal{O}_{\tilde{P}}\oplus \tilde{\delta}^\vee$. (see  \cite{bhpv} Chapter V.22)
	
	For any contraction $\varphi: \tilde{P}\rightarrow P'$ of $(-1)$-curves and $R'=\varphi(\tilde{R})$, the double cover $\tilde{\pi}$ induces a double cover $S'\rightarrow P'$, which is  determined by $(R',\delta')$. We call $(R',\delta')$ the image of $(\tilde{R},\tilde{\delta})$.
	
	\begin{lemma}[\cite{Xiao91,Xiao92}]
		\label{lemma for psi}
		There exists a contraction of a ruled surfaces $\psi: \tilde{P}\rightarrow P$:
		$$\xymatrix{\tilde{P} \ar[rr]^{\psi} \ar[rrd]^{\tilde{h}} && P\ar[d]^h\\
			~ && B}
		$$
		such that $P$ is a geometrical ruled surface (i.e. any fiber of $h$ is $\mathbb{P}^1$), the singularities of $R$ are at most of multiplicity $g+2$, and the self-intersection $R^2$ is the smallest among all such choices, where $(R,\delta)$ is the image of $(\tilde{R},\tilde{\delta})$ in $P$.
	\end{lemma}

	The contraction $\psi$ in \autoref{lemma for psi} can be decomposed  into $\tilde{\psi}:\tilde{P}\rightarrow \hat{P}$ and $\hat{\psi}: \hat{P}\rightarrow P$, where $\tilde{\psi}$ and $\hat{\psi}$ are composed respectively of resolutions of negligible and non-negligible singularities of $R$, cf. \cite[Def\,4]{Xiao91} or \cite[Def\,2.1]{LZ17}.	
	Let $(\hat{R},\hat{L})$ be the image of $(\tilde{R},\tilde{L})$ in $\hat{P}$. Let $\hat{\psi}=\hat{\psi}_1\circ \dotsm \circ \hat{\psi}_t$ be the decomposition of $\hat{\psi}$, where $\hat{\psi}_i: \hat{P}_i\rightarrow \hat{P}_{i-1}$ is the blow-up at $y_{i-1}$, $\hat{P}_0=P$ and $\hat{P}_t=\hat{P}$.
	Let $\hat{R}_i$ be the image of $\hat{R}$ in $\hat{P}_i$. We remark that the decomposition of $\hat{\psi}$ is not unique.
	If $y_{i-1}$ is a singular point of $\hat{R}_{i-1}$ of odd multiplicity $2k+1~(k\geq 1)$ and there is a unique singular point $y$ of $\hat{R}_i$ on the exceptional curve $\hat{\mathcal{E}}_i$ of multiplicity $2k+2$, then we always assume that $\hat{\psi}_{i+1}: \hat{P}_{i+1}\rightarrow \hat{P}_i$ is the standard blow-up at $y_i=y$. We call such a pair $(y_{i-1},y_i)$ a singularity of type  $(2k+1\rightarrow 2k+1)$ and call $y_{i-1}$ (resp. $y_i$) the first (resp. second) component of such a singularity.
	
	\begin{definition}
		For any singular fiber $F$ of $f$ and $3\leq i\leq g+2$, the $i$-th singularity index of $F$ is defined as follows (with respect to the contraction $\psi$):
		
		(1) if $i$ is odd, $s_i(F)$ equals the number of $(i\rightarrow i)$ type  singularities of $R$ over the image $f(F)$;
		
		(2) if   $i$ is even, $s_i(F)$ equals the number of singularities of multiplicity $i$ or $i+1$ of $R$ over the  image $f(F)$, neither belonging to the second component of $(i-1\rightarrow i-1)$ type singularities nor to the first component of $(i+1\rightarrow i+1)$ type singularities.
	\end{definition}

	Here we give an example to help to understand the above singularity indices.
	\begin{example}\label{example-singularity}
		Suppose that $F$ is a singular fiber of $f$, and $\Gamma$ is its corresponding fiber in the geometrical ruled surface $h:\,P \to B$.
		Let $t$ be a local coordinate of $B$ around $b=h(\Gamma)$,
		and $x$ be an affine fiber-coordinate of the $\bbp^1$-bundle over $B$.
		
		\begin{enumerate}
			\item Suppose that the local equation of $R$ over $b$ is $(t^{2k}-x^{2k})(x^{2g+2-2k}-1)=0$ with $2\leq k\leq \big[\frac{g+1}{2}\big]$.
			Then $p$ is the unique singularity of $R$ on $\Gamma$ of multiplicity $2k$.
			In this case, $s_{2k}(F)=1$ and $s_i(F)=0$ for all other $i\geq 3$.
			
			\item Suppose that the local equation of $R$ over $b$ is $(t^{2(2k+1)}-x^{2k+1})(x^{2g+1-2k}-1)=0$ with $1\leq k\leq \big[\frac{g}{2}\big]$.
			Then $p$ is  a singularity of $R$ on $\Gamma$ contained in a singularity of type $(2k+1 \to 2k+1)$.
			In this case, $s_{2k+1}(F)=1$ and $s_i(F)=0$ for all other $i\geq 3$.

			\item Suppose that $g$ is odd, and that the local equation of $R$ over $b$ is $t(t^{g+1}-x^{2(g+1)})=0$.
			Then $p$ is also the unique singularity of $R$ on $\Gamma$.
			In this case, $s_{g+2}(F)=1$ and $s_i(F)=0$ for all other $i\geq 3$.
		\end{enumerate}
	\end{example}

We remark that the minimality of $R^2$ implies that $s_{g+2}(F)=0$ if $g$ is even (c.f. \cite[Thm\,2.6 and its proof]{LZ17})

	Let $K_{\hat{P}/B}=K_{\hat{P}}-\hat{h}^*K_B$ and $R'={\hat{R}}\setminus {\hat{V}}$, where $\hat{V}$ is the union of isolated vertical $(-2)$-curves in $\hat{R}$.
	Here an irreducible curve $C\subset \hat{R}$ is said to be   isolated in $R$ if there is no other irreducible curve $C'\subset \hat{R}$ such that $C\cap C'\neq\emptyset$.
	We define
	
	$$s_2\triangleq (K_{\hat{P}/B}+R')\cdot R' \qquad \text{and} \qquad s_i \triangleq \sum_{F ~\text{is singular}} s_i(F), ~~3\leq i\leq g+2.$$	
	Note that  $s_i$ is nonnegative for $i\geq 3$ by definition, but the singularity index $s_2$ might be negative,
	cf. \cite{ggl25}.
	
	\begin{lemma}[\cite{Xiao91,Xiao92}]
		These singularity indices $s_i$'s defined above are independent on the choices of $\psi$ in \autoref{lemma for psi}.
	\end{lemma}

	According to \cite[page 604]{Xiao91},  we can write
	$R\sim -(g+1)K_{P/B}+nF$    where `$\sim$' stands for numerical equivalence, and \begin{equation}\label{eqn-def-n}
		n=\frac{R^2}{4(g+1)},
	\end{equation}
	which is an integer.
	The following formulas for hyperelliptic fibrations are due to Xiao; we refer to \cite[Thm\,1]{Xiao91} or \cite[Thm\,2.6]{LZ17} for a proof.
	
	\begin{theorem}	\label{key lemma of Xiao}
		Let $f:S\rightarrow B$ be a hyperelliptic fibration of  genus $g\geq 2$, and let $s_i$'s be the singularity indices as above. Then
		$$\begin{aligned}
		\chi_f&\,=\frac{1}{2}gn-\sum_{k=1}^{[\frac{g+1}{2}]}k^2s_{2k+1}-\sum_{k=2}^{[\frac{g+1}{2}]}\frac{k(k-1)}{2}s_{2k}\\
		&\,=\frac{gs_2+(g^2-2g-1)s_{g+2}}{4(2g+1)}+\sum_{k=1}^{[\frac{g}{2}]}\frac{k(g-k)}{2g+1}s_{2k+1}+\sum_{k=2}^{[\frac{g+1}{2}]}\frac{k(g-k+1)}{2(2g+1)}s_{2k};\\[4pt]
		K_f^2&\,=(2g-2)n+s_{g+2}-\sum_{k=1}^{[\frac{g+1}{2}]}(2k-1)^2s_{2k+1}-\sum_{k=2}^{[\frac{g+1}{2}]}2(k-1)^2s_{2k}\\
		&\,=\frac{(g-1)\big(s_2+(3g+1)s_{g+2}\big)}{2g+1}+\sum_{k=1}^{[\frac{g}{2}]}\frac{12k(g-k)-2g-1}{2g+1}s_{2k+1}+\sum_{k=2}^{[\frac{g+1}{2}]}\frac{6k(g-k+1)-4g-2}{2g+1}s_{2k};\\[4pt]
		e_f&\,=s_2-2s_{g+2}+\sum_{k=1}^{[\frac{g}{2}]}s_{2k+1}+2\sum_{k=2}^{[\frac{g+1}{2}]}s_{2k}.
		\end{aligned}$$
	\end{theorem}
	
	Suppose that $f$ is not locally trivial. 	
	Then the above formulas imply in particular that
	\begin{equation}\label{eqn-4-2}
	\begin{aligned}
	&\left(\lambda_f-\frac{4(g-1)}{g}\right)\chi_f =K_f^2-\frac{4(g-1)}{g}\chi_f\\
	&=\frac{g^2-1}{g}s_{g+2}+\sum_{k=1}^{[\frac{g}{2}]}\frac{4k(g-k)-g}{g}s_{2k+1}+\sum_{k=2}^{[\frac{g+1}{2}]}\frac{2k(g-k+1)-2g}{g}s_{2k}.
	\end{aligned}
	\end{equation}
	
	\begin{corollary}[\cite{Xiao92,lt13}]
		Let $f:S\rightarrow B$ be a not locally trivial hyperelliptic fibration of  genus $g\geq 2$. Then the slope $\lambda_f$ satisfies
		\begin{equation}\label{eqn-slope-hyperelliptic}
			\frac{4(g-1)}{g} \leq \lambda_f \leq \left\{\begin{aligned}
			&12-\frac{8g+4}{g^2}, &~& \text{~if $g$ is even}\\
			&12-\frac{8g+4}{g^2-1}, && \text{~if $g$ is odd}
			\end{aligned}\right\}
			<12.
		\end{equation}
	\end{corollary}
	
	If moreover the base curve $B$ is an elliptic curve, then the upper bound of $\lambda_f$ can be improved.
	In this case, $g(B)=1$ and hence
	the relative invariants of $f$ equal the corresponding invariants of the surface $S$, i.e.,
	\begin{equation}\label{eqn-2-1}
		K_f^2=K_S^2, \qquad \chi_f=\chi=\chi(\mathcal{O}_S),\qquad e_f=e(S).
	\end{equation}
	By the Miyaoka-Yau inequality together with \eqref{eqn-2-1}, one obtains
	\begin{equation}\label{eqn-4-1}
	\lambda_f=\frac{K_f^2}{\chi_f} \leq 9.
	\end{equation}	
	
To end this section, we mention that by
\eqref{eqn-4-2}together with \eqref{eqn-4-1} and \eqref{eqn-slope-hyperelliptic}, one obtains
	\begin{eqnarray}
		\chi_f &\hspace{-2mm}\geq\hspace{-2mm}&
		\frac{g^2-1}{5g+4}s_{g+2}+\sum_{k=1}^{[\frac{g}{2}]}\frac{(4k-1)g-4k^2}{5g+4}s_{2k+1}+
		\sum_{k=2}^{[\frac{g+1}{2}]}\frac{2(k-1)(g-k)}{5g+4}s_{2k}, \quad \text{if~}g(B)=1,\label{(1.25g+1)chi_f}\\
		\chi_f &\hspace{-2mm}>\hspace{-2mm}&
		\frac{g^2-1}{8g+4}s_{g+2}+\sum_{k=1}^{[\frac{g}{2}]}\frac{(4k-1)g-4k^2}{8g+4}s_{2k+1}+
		\sum_{k=2}^{[\frac{g+1}{2}]}\frac{2(k-1)(g-k)}{8g+4}s_{2k}, \quad \text{if~}g(B)\geq 2.\qquad\quad\label{eqn-3-44}
	\end{eqnarray}

	\section{Quadratic upper bounds on hyperelliptic Albanese fibrations}\label{sec-upper-hyperelliptic}
	The main purpose of this section is to prove \autoref{thm-upper1}.
	We will prove the following upper bounds (using the relative invariants) on the genus $g$ for any not locally trivial hyperelliptic fibration, not necessarily the Albanese fibration.
	We remark once more that the relative invariants of $f$ equal the corresponding invariants of the surface $S$ (cf. \eqref{eqn-2-1}) if the base $B$ is an elliptic curve.
	\begin{theorem}\label{theorem upper bound for g}
		Let $f:\,S \to B$ be a relatively minimal hyperelliptic fibration of genus $g\geq 2$.
		Suppose that $f$ is not locally trivial.
		\begin{enumerate}
			\item If $\lambda_f\leq 4$, then
			$$g\leq \frac{4\chi_f+4}{2+(4\chi_f-K_f^2)} \leq 2\chi_f+2.$$
			
			\item If $\lambda_f>4$, then
			$$g\leq \left\{\begin{aligned}
			&2\chi_f+1, && \text{if $g(B)=0$};\\
			&\frac{25}{4}\chi_f^2+\frac{19}{2}\chi_f+2, &\quad& \text{if $g(B)=1$};\\
			&16\chi_f^2+14\chi_f+2, &&\text{if $g(B)\geq 2$}.
			\end{aligned}\right.$$
		\end{enumerate}\vspace{2mm}
	\end{theorem}
	We first prove \autoref{thm-upper1} based on \autoref{theorem upper bound for g}.
	\begin{proof}[Proof of \autoref{thm-upper1}]
		Let $f:\,S \to B$ be the Albanese fibration of $S$ and let $g$ be the genus of a general fiber of $f$.
		Since $g(B)=q(S)=1$,   $f$ must be not locally trivial; otherwise $\chi(\mathcal{O}_S)=\chi_f=0$ which is absurd.
		Thus by \eqref{eqn-2-1} together with \autoref{theorem upper bound for g}\,(2),
		one obtains \eqref{eqn-upper-1}.
	\end{proof}


	\begin{remark}
		(1) If $\lambda_f=4$ and $g(B)=1$, then the upper bound $g\leq 2\chi_f+2$ has already been obtained by Ishida \cite{Ish05}.
		

	(2)	
		We will construct examples in \autoref{sec-example} showing that the quadratic upper bounds can not be improved into a linear one.
		
	\end{remark}

The proof of \autoref{theorem upper bound for g} relies on the following technical result.
\begin{theorem}\label{thm-sharp-upper-bound}
	Let $f:\,S \to B$ be a relatively minimal hyperelliptic fibration of genus $g\geq 2$, and $n=\frac{R^2}{4(g+1)}$ be an integer as in \eqref{eqn-def-n}.
	Suppose that $f$ is not locally trivial with slope $\lambda_f>4$ and $g(B)\geq 1$.
	Then
\begin{equation}\label{eqn-maximal-n}
 1\leq  n\leq \Big(\frac{\lambda_f-4}{2}+\frac{\lambda_f}{g-1}\Big)\chi_f,
\end{equation}
and
	\begin{equation}\label{eqn-6-1}
	g\leq \frac{(\lambda_f-4)^2}{4n}\chi_f^2+(\lambda_f-4+\frac{\lambda_f}{2n})\chi_f+n+1.
	\end{equation}	
	Moreover, equality in \eqref{eqn-6-1} holds if and only if  there exists exactly one $s_{2k}=1$ for some $k\geq 2$ and $s_i=0$ for all $i\geq 3, i\neq 2k$.	
\end{theorem}
We first prove \autoref{theorem upper bound for g} based on \autoref{thm-sharp-upper-bound} in the following, and then prove \autoref{thm-sharp-upper-bound}.
\begin{proof}[Proof of \autoref{theorem upper bound for g}]
        (1).  Note that if $s_i=0$ for all $i\geq 3$,  then by \autoref{key lemma of Xiao} we get $\chi_f=\frac{1}{2}gn$ and consequently $g=\frac{2\chi_f}{n}\leq 2\chi_f<2\chi_f+2$. So we can assume that $s_i\neq 0$ for some $i\geq 3$.

         Since $\lambda_f\leq 4$, i.e., $K_f^2\leq 4\chi_f$,
        it follows that $t:=4\chi_f-K_f^2\geq 0$.
        By \autoref{key lemma of Xiao} one gets
        $$\begin{aligned}
        \chi_f-\frac{gt}{4}&\,=\frac{g}{4}\left(K_f^2-\frac{4(g-1)}{g}\chi_f\right)\\
        &\,=\frac{g^2-1}{4}s_{g+2}+\sum_{k=1}^{[\frac{g}{2}]}\frac{(4k-1)g-4k^2}{4}s_{2k+1}+
        \sum_{k=2}^{[\frac{g+1}{2}]}\frac{(k-1)(g-k)}{2}s_{2k}\\
        &\,\geq \min\left\{\frac{g^2-1}{4},~ \min_{1\leq k \leq [\frac{g}{2}]} \frac{(4k-1)g-4k^2}{4},~\min_{2\leq k \leq [\frac{g+1}{2}]} \frac{(k-1)(g-k)}{2}\right\}\\
        &\,= \frac{g-2}{2}.
        \end{aligned}$$
        Thus $g\leq \frac{4\chi_f+4}{2+t} \leq 2\chi_f+2$ as required.

		\vspace{2mm}
		(2).
		If $g(B)=0$, then
		$$\chi_f=\chi(\mathcal{O}_S)+(g-1)=p_g-q+g \geq \frac{g-1}{2}.$$
		The inequality above follows from the inequality $q\leq \frac{g+1}{2}$ by Xiao \cite{xiao87-2} and the non-negativity of $p_g$.
		Hence $g\leq 2\chi_f+1$ as required.
		
		\vspace{2mm}
		We consider next the case when $g(B)\geq 1$.
 By \autoref{thm-sharp-upper-bound} above, we have
\begin{equation}\label{eqn-3-11}
 g\leq \frac{(\lambda_f-4)^2}{4n}\chi_f^2+(\lambda_f-4+\frac{\lambda_f}{2n})\chi_f+n+1, \qquad
 n\leq \big( \frac{\lambda_f-4}{2} +\frac{\lambda_f}{g-1}\big) \chi_f.
\end{equation}
If $g(B)=1$, then $\lambda_f\leq 9$ by \eqref{eqn-4-1}, and hence
$$g\leq \max_{1 \leq n\leq \big( \frac{\lambda_f -4}{2} +\frac{\lambda_f}{g-1}\big) \chi_f} \left\{g(n):= \frac{25}{4n}\chi_f^2+\Big(5+\frac{9}{2n}\Big)\chi_f+n+1 \right\}.$$
Suppose that
\begin{equation}\label{eqn-suppose-2}
g>\frac{25}{4}\chi_f^2+\frac{19}{2}\chi_f+2=g(1).
\end{equation}
Since $\chi_f\geq 1$, we have $g\geq 18$ and thus $n\leq 3\chi_f$.
Note that as $n$ increases, $g(n)$ first decreases and then increases. Hence we have
$$g\leq \max_{1 \leq n\leq \big( \frac{\lambda_f -4}{2} +\frac{\lambda_f}{g-1}\big) \chi_f} \{g(n) \}\leq \max\{g(1), g(3\chi_f)\}=\max\{g(1), \frac{121}{12}\chi_f+\frac{5}{2}\}=g(1),$$
which  contradicts to \eqref{eqn-suppose-2}.

If $g(B)=2$, the argument is almost the same as the case when $g(B)=1$, except that we have to replace
\eqref{eqn-4-1} by \eqref{eqn-slope-hyperelliptic}; the details are omitted here.
\end{proof}

We now prove the technical result \autoref{thm-sharp-upper-bound}.
\begin{proof}[{Proof of \autoref{thm-sharp-upper-bound}}]
	We first prove \eqref{eqn-6-1}. By \autoref{key lemma of Xiao}, we have
%
	\begin{equation}\label{eqn-7-2}
	(\lambda_f-4)\chi_f+2n=K_f^2-4\chi_f+2n=(2g+2)s_{g+2}+\sum_{k=1}^{[\frac{g}{2}]}(4k-1) s_{2k+1}+\sum_{k=2}^{[\frac{g+1}{2}]}2(k-1)s_{2k}.
	\end{equation}
	Hence
	\begin{eqnarray}
	&&(\lambda_f-4)^2 \chi_f^2+(4n+3)(\lambda_f-4)\chi_f+4n^2+6n\nonumber\\
	&=& \big((\lambda_f-4)\chi_f+2n\big)^2+3\big((\lambda_f-4)\chi_f+2n\big)\nonumber\\
	&=& \Big((2g+2)s_{g+2}+\sum_{k=1}^{[\frac{g}{2}]}(4k-1) s_{2k+1}+\sum_{k=2}^{[\frac{g+1}{2}]}2(k-1)s_{2k}\Big)^2\nonumber\\
	 &&+3\Big((2g+2)s_{g+2}+\sum_{k=1}^{[\frac{g}{2}]}(4k-1) s_{2k+1} +\sum_{k=2}^{[\frac{g+1}{2}]}2(k-1)s_{2k}\Big)\nonumber\\
	&\geq&(g+1)(4g+10)s_{g+2}+\sum_{k=1}^{[\frac{g}{2}]}(16k^2+4k-2) s_{2k+1}+\sum_{k=2}^{[\frac{g+1}{2}]} (4k^2-2k-2) s_{2k}\nonumber\\
	&\geq&(g+1)(2g+4)s_{g+2}+\sum_{k=1}^{[\frac{g}{2}]}(8k^2+4k-1) s_{2k+1}+\sum_{k=2}^{[\frac{g+1}{2}]} (4k^2-2k-2) s_{2k}. \label{eqn-7-3}
	\end{eqnarray}

	On the other hand, by \autoref{key lemma of Xiao}, we have
	\begin{align}
	\chi_f&\,=\frac{gs_2+(g^2-2g-1)s_{g+2}}{4(2g+1)}+\sum_{k=1}^{[\frac{g}{2}]}\frac{k(g-k)}{2g+1}s_{2k+1}+
	\sum_{k=2}^{[\frac{g+1}{2}]}\frac{k(g-k+1)}{2(2g+1)}s_{2k}\nonumber\\
	&\,=\frac{g}{4(2g+1)}\Big(e_f-\sum_{k=1}^{[\frac{g}{2}]} s_{2k+1}-\sum_{k=2}^{[\frac{g+1}{2}]}2s_{2k}\Big)+\frac{(g^2-1)s_{g+2}}{4(2g+1)}
+\sum_{k=1}^{[\frac{g}{2}]}\frac{k(g-k)}{2g+1}s_{2k+1}+
	\sum_{k=2}^{[\frac{g+1}{2}]}\frac{k(g-k+1)}{2(2g+1)}s_{2k}\nonumber\\ &\,=\frac{1}{8}\Big(e_f+(2g+2)s_{g+2}\Big)+\sum_{k=1}^{[\frac{g}{2}]}\big(\frac{k}{2}-\frac{1}{8}\big)s_{2k+1}
+\sum_{k=2}^{[\frac{g+1}{2}]}\big(\frac{k}{4}-\frac{1}{4}\big)s_{2k}\nonumber\\
	&\quad	-\frac{1}{2g+1}\left(\frac{1}{8}\Big(e_f+(g+1)(2g+4)s_{g+2}\Big)+\sum_{k=1}^{[\frac{g}{2}]}\big(\frac{2k^2+k}{2}-\frac{1}{8}\big)s_{2k+1}+
	\sum_{k=2}^{[\frac{g+1}{2}]}\big(\frac{2k^2-k}{4}-\frac{1}{4}\big)s_{2k}\right)
	\nonumber
	\end{align}	
	Combining this with \eqref{eqn-7-2}, we get
	\begin{align}
	2n &\,=e_f+(2g+2)s_{g+2}+\sum_{k=1}^{[\frac{g}{2}]}(4k-1) s_{2k+1}+\sum_{k=2}^{[\frac{g+1}{2}]}2(k-1)s_{2k}-8\chi_f\nonumber\\
	&\,=\frac{1}{2g+1}\bigg(e_f+(g+1)(2g+4)s_{g+2}+\sum_{k=1}^{[\frac{g}{2}]}(8k^2+4k-1)s_{2k+1}+\sum_{k=2}^{[\frac{g+1}{2}]}(4k^2-2k-2)s_{2k}\bigg). \label{eqn-7-5}
	\end{align}
	Note that the left hand side of the above equality is an integer and that the right hand side of the above equality is positive.
Hence we have
	\begin{equation}\label{eqn-7-6}
	e_f+(g+1)(2g+4)s_{g+2}+\sum_{k=1}^{[\frac{g}{2}]}(8k^2+4k-1)s_{2k+1}+\sum_{k=2}^{[\frac{g+1}{2}]}(4k^2-2k-2)s_{2k}=2n(2g+1).
	\end{equation}
	Now combining \eqref{eqn-7-3} and \eqref{eqn-7-6}, we get	
	\begin{align}
	&\,(\lambda_f-4)^2 \chi_f^2+\big(4n(\lambda_f-4)+2\lambda_f\big)\chi_f+4n^2+6n\nonumber\\
	&\,=e_f+(\lambda_f-4)^2 \chi_f^2+(4n+3)(\lambda_f-4)\chi_f+4n^2+6n \nonumber\\
	&\,\geq e_f+(g+1)(2g+4)s_{g+2}+\sum_{k=1}^{[\frac{g+1}{2}]}(8k^2+4k-1)s_{2k+1}+\sum_{k=2}^{[\frac{g+1}{2}]}(4k^2-2k-2)s_{2k}\nonumber\\
	&\,=2n(2g+1). \label{eqn-7-7}
	\end{align}
	Hence we get
	$$g\leq \frac{(\lambda_f-4)^2}{4n}\chi_f^2+\big(\lambda_f-4+\frac{\lambda_f}{2n}\big)\chi_f+n+1.$$
	
	Finally,  the equality in \eqref{eqn-6-1} holds if and only if equality in  \eqref{eqn-7-3} holds,
	and if and only if there exists exactly one $s_{2k}=1$ for some $k\geq 2$ and $s_i=0$ for all $i\geq 3, i\neq 2k$.	

\vspace{1ex}
Now we prove \eqref{eqn-maximal-n}.
According to \eqref{eqn-4-2} and \eqref{eqn-7-5}, we get
  \begin{equation*}
	\begin{aligned}
	\frac{(\lambda_f-4)g+4}{g}\chi_f
    & \geq \frac{g-1}{2g}\left( (2g+2)s_{g+2}+\sum_{k=1}^{[\frac{g}{2}]}(4k-1) s_{2k+1}+\sum_{k=2}^{[\frac{g+1}{2}]}2(k-1)s_{2k} \right) \\
	&\geq \frac{g-1}{2g}\big( (\lambda_f-4)\chi_f+2n \big).
	\end{aligned}
	\end{equation*}
This proves \eqref{eqn-maximal-n}.
\end{proof}

\begin{remark}
We will construct  examples reaching the equality  in \eqref{eqn-6-1}; see \autoref{thm-example-2}.
\end{remark}


If $\chi_f$ is large, one can obtain a better upper bound on $g$ as follows.
\begin{lemma}\label{lemma-maximal-g(n)}
Let $f:\,S \to B$ be a relatively minimal hyperelliptic fibration of genus $g\geq 2$, and $n=\frac{R^2}{4(g+1)}$ be an integer as in \eqref{eqn-def-n}.
	Suppose that $f$ is not locally trivial with slope $\lambda_f>4$ and $g(B)\geq 1$.
 Set
  $g(n):=\frac{(\lambda_f-4)^2}{4n}\chi_f^2+\big(\lambda_f-4+\frac{\lambda_f}{2n}\big)\chi_f+n+1$.
 If  $g\geq 25$ and  $\chi_f\geq 4$, then we have
  $$g(1)\geq g(2)\geq \max_{1 \leq n\leq \big( \frac{\lambda_f -4}{2} +\frac{\lambda_f}{g-1}\big) \chi_f} \{g(n)\}.$$
 \end{lemma}

 \begin{proof}
 (1) Since $\chi_f\geq 1$ and $\lambda_f\geq 4$, we have $\frac{\lambda_f}{2}\chi_f+1\geq \frac{\lambda_f}{4}\chi_f+2$ and thus $g(1)\geq g(2)$. So we only need to prove
  $$g(2)\geq \max\limits_{1 \leq n\leq \big( \frac{\lambda_f -4}{2} +\frac{\lambda_f}{g-1}\big) \chi_f} \{g(n)\}.$$

  Since $n\leq \big(\frac{\lambda_f-4}{2}+\frac{\lambda_f}{g-1}\big)\chi_f$, $\lambda_f<12$ (see \eqref{eqn-slope-hyperelliptic}) and   we have assumed $g\geq 25$, we get  $n<\frac{\lambda_f-3}{2}\chi_f$.  Note that as $n$ increases, $g(n)$ first decreases and then increases.
Hence  we only need to show $g(2)\geq g\big(\frac{\lambda_f-3}{2}\chi_f\big)$.

$$g(2)-g\big(\frac{\lambda_f-3}{2}\chi_f\big)
=\Big( \frac{\chi_f}{8}-\frac{1}{2(\lambda_f-3)}\Big)(\lambda_f-4)^2\chi_f
+\frac{6-\lambda_f}{4}\chi_f+2-\frac{\lambda_f}{\lambda_f-3}.$$

Since  we have assumed $\lambda_f\geq 4$ and  $\chi_f\geq 4$, we have $\frac{\chi_f}{8}\geq \frac{1}{2}\geq \frac{1}{2(\lambda_f-3)}$.

 (i) If $\lambda_f\leq 6$, then we have
 $$g(2)-g\big(\frac{\lambda_f-3}{2}\chi_f\big)\geq 6-\lambda_f+2-\frac{\lambda_f}{\lambda_f-3}
 =4-\left( (\lambda_f-3)+\frac{3}{\lambda_f-3}\right)\geq 0.$$

  (ii) If $\lambda_f>6$, then we have $\lambda_f-3\geq 3$ and $2-\frac{\lambda_f}{\lambda_f-3}\geq 0$. Hence we have
  $$g(2)-g\big(\frac{\lambda_f-3}{2}\chi_f\big)\geq
  \left( \big(\frac{1}{2}-\frac{1}{6}\big)(\lambda_f-4)^2+\frac{1}{2}-\frac{\lambda_f-4}{4}\right)\chi_f>0.$$

  Therefore if $g\geq 25$ and $\chi_f\geq 4$,  we always have
  $g(2)\geq g\big(\frac{\lambda_f-3}{2}\chi_f\big)$.
 \end{proof}

\begin{proposition}\label{pro-4-1}
 	Let $f:\,S \to B$ be a relatively minimal hyperelliptic fibration of genus $g\geq 2$.
	Suppose that $f$ is not locally trivial with slope $\lambda_f\geq 4$ and $g(B)\geq 1$.
	   If $g\geq 25$ and $\chi_f\geq 4$, then we have
\begin{equation}\label{eqn-6-12}
g\leq \left\{\begin{aligned}
			&\frac{(\lambda_f-4)^2}{4}\chi_f^2+(\frac{3}{2}\lambda_f-4)\chi_f+2, &\quad& \text{if $g$ is even};\\
			&\frac{(\lambda_f-4)^2}{8}\chi_f^2+(\frac{5}{4}\lambda_f-4)\chi_f+3, &&\text{if $g$ is odd}.
			\end{aligned}\right.
\end{equation}
	Moreover, equality in \eqref{eqn-6-12} holds if and only if  there exists exactly one $s_{2k}=1$ for some $k\geq 2$ and $s_i=0$ for all $i\geq 3, i\neq 2k$.	

\end{proposition}
\begin{proof}
  By \autoref{thm-sharp-upper-bound} and \autoref{lemma-maximal-g(n)}, we only need to prove $n\geq 2$ if $g$ is odd.
  Here we use notations as in \autoref{sec-hyperelliptic}. Recall that
  $R\sim -(g+1)K_{P/B}+n\Gamma$ is an even divisor. If $g$ is odd, then $-(g+1)K_{P/B}$ is an even divisor and thus $n\Gamma$ is also an even divisor. Since the N\'eron-Severi group  $\text{Num}(P)$ is generated by a section and a fibre $\Gamma$ of $h:\, P\rightarrow B$, we see $n$ is even. Since $n$ is a positive integer by \eqref{eqn-7-5}, we get $n\geq 2$.
\end{proof}

Since we have $\lambda_f\leq 9$ if $g(B)=1$ (see \eqref{eqn-4-1}) and $\lambda_f <12$ if $g(B)\geq 2$ (see \eqref{eqn-slope-hyperelliptic}), we get the following.
\begin{corollary}
Under the same assumptions as in \autoref{pro-4-1} the following hold.

 (1) If $g(B)=1$, we have
 \begin{equation}\label{eqn-6-13}
g\leq \left\{\begin{aligned}
			&\frac{25}{4}\chi_f^2+\frac{19}{2}\chi_f+2, &\quad& \text{if $g$ is even};\\
			&\frac{25}{8}\chi_f^2+\frac{29}{4}\chi_f+3, &&\text{if $g$ is odd}.
			\end{aligned}\right.
\end{equation}

  (1) If $g(B)\geq 2 $, we have
 \begin{equation}\label{eqn-6-14}
g\leq \left\{\begin{aligned}
			&16\chi_f^2+14\chi_f+2, &\quad& \text{if $g$ is even};\\
			&8\chi_f^2+11 \chi_f+3, &&\text{if $g$ is odd}.
			\end{aligned}\right.
\end{equation}
\end{corollary}

	\section{Upper bound on hyperelliptic Albanese fibrations with $p_g=q=1$}\label{sec-pg=q=1}
	The aim in this section is to prove \autoref{thm-upper-chi=1}.
	So we always assume in this section that $S$ is a minimal irregular surface of general type  with $p_g=q=1$
	and  its Albanese fibration $f:S\rightarrow B$ is hyperelliptic.
	In this case, $g(B)=q=1$ and $f$ is not locally trivial otherwise $0=\chi_f=\chi(\mathcal{O}_S)$.
	By \eqref{eqn-2-1}, the relative invariants equal the corresponding invariants of $S$.
	We use the singularity indices $s_i$'s introduced in \autoref{sec-hyperelliptic}.
	Since it is unknown whether $s_2$ is nonnegative, we divide the proof into two cases, depending on whether $s_2<0$ or $s_2\geq 0$.
	The proof of \autoref{thm-upper-chi=1} will be completed in \autoref{pro for s_2<0,p_g=1}, \autoref{prop-s_s>0} and \autoref{prop-s_s>0-2}.
	\begin{proposition}\label{pro for s_2<0,p_g=1}
	Let $S$ be a minimal   surface of general type with $p_g(S)=q(S)=1$.
		Suppose that the Albanese fibration $f:\,S\to B$ of $S$ is hyperelliptic of genus $g$.
		If $s_2<0$, then $g\leq 5$.
	\end{proposition}	
	\begin{proof}
		We first claim that
	\begin{claim}\label{claim-1}
		If $s_2<0$, then we have $K_S^2\leq 7$.
	\end{claim}
        \begin{proof}[Proof of \autoref{claim-1}]
        	If $s_2<0$, using notations in \autoref{sec-hyperelliptic},
        	we see that $\hat{R}$ must contain some isolated curve $C$ with $C^2\neq -2$. So there must be some smooth rational curves $\widetilde{C}$ contained in fibers of $f$.
        	Assume $\widetilde{C}^2=-n$, by \cite[Theorem\,1.1]{Miy84} we have
        	$$\frac{3}{2}\leq \frac{(n+1)^2}{3n}\leq e(S)-\frac{1}{3}K_S^2=\frac{4}{3}(9\chi-K_S^2).$$
        	Hence we get $K_S^2\leq 9-\frac{9}{8}$, i.e. $K_S^2\leq 7$ since $K_S^2$ is an integer.
        \end{proof}
	Come back to the proof of \autoref{pro for s_2<0,p_g=1}.  Since $s_2$ is an integer and we have assumed that $s_2<0$, we see that $s_2\leq -1$.	By \autoref{key lemma of Xiao}, we have
		$$\begin{aligned}
		&K_f^2-\frac{8g-14}{g-1}\chi_f\\ =\,&-\frac{g-2}{2(g-1)}s_2+\frac{g^2+2g-5}{2(g-1)}s_{g+2}+
		\sum_{k=1}^{[\frac{g}{2}]}\left(\frac{2k(g-k)}{g-1}-1\right)s_{2k+1}+\sum_{k=3}^{[\frac{g+1}{2}]}\left(\frac{k(g-k+1)}{g-1}-2\right)s_{2k}\\
		\geq\,&\frac{g-2}{2(g-1)}.
		\end{aligned}$$
		Hence
		$$\frac{g-2}{2(g-1)} \leq K_f^2-\frac{8g-14}{g-1}\chi_f \leq 7-\frac{8g-14}{g-1}=\frac{7-g}{g-1}.$$
		It follows that $g\leq \frac{16}{3}$, i.e., $g\leq 5$ as required.
	\end{proof}

	In the remaining part of this section, we assume that  $s_2\geq 0$.
	We first claim that
	\begin{claim}
	Let $S$ be a minimal   surface of general type with $p_g(S)=q(S)=1$.
		Suppose that the Albanese fibration $f:\,S\to B$ of $S$ is hyperelliptic of genus $g$.
 Suppose that $s_2\geq 0$. Then
		\begin{equation}\label{eqn-3-1}
		\left\{\begin{aligned}
		&\text{if $g\geq 6$,} &\quad& \text{then $s_{g+2}=0$;}\\
		&\text{if $g\geq 11$,} && \text{then $s_{2k+1}=0,\quad \forall~k\geq 2$;}\\
		&\text{if $g\geq 13$,} && \text{then $s_{2k}=0,\quad \forall~k\geq 6$.}
		\end{aligned}\right.
		\end{equation}
	\end{claim}
\begin{proof}
	If $s_{g+2}>0$, then by \eqref{eqn-2-1} and \eqref{(1.25g+1)chi_f} we have
	$$1=\chi=\chi_f \geq \frac{g^2-1}{5g+4}s_{g+2} \geq \frac{g^2-1}{5g+4},\quad \Longrightarrow\quad g^2-5g-5 \leq 0.$$
	Hence $g<6$. This shows that $s_{g+2}=0$ if $g\geq 6$.
	The other two inequalities can be proved similarly using \eqref{(1.25g+1)chi_f}, and are left to the readers.
\end{proof}

	\begin{proposition}\label{prop-s_s>0}
	Let $S$ be a minimal   surface of general type with $p_g(S)=q(S)=1$.
		Suppose that the Albanese fibration $f:\,S\to B$ of $S$ is hyperelliptic of genus $g$.
	If $g\geq 13$, then there is only one possible case for $(K_S^2,g)$: $K_S^2=8, g=14$.
	In particular, $g\leq 14$.
   \end{proposition}
\begin{proof}
	Since $g\geq 13$, by \eqref{eqn-3-1} together with \autoref{key lemma of Xiao},
	\begin{eqnarray}
		2g+1&=&\frac{g}{4}s_2+(g-1)s_3+(g-1)s_4+\frac{3}{2}(g-2)s_6+2(g-3)s_8+\frac{5}{2}(g-4)s_{10}, \label{eqn-3-2}\\
		\frac{gn}{2}-1&=&s_3+s_4+3s_6+6s_8+10s_{10}.\label{eqn-3-3}
	\end{eqnarray}
	
	(1) We show first that $s_{10}=0$. Otherwise by \eqref{eqn-3-2}, $s_{10}=1$ and
	$$11-\frac{g}{2}=\frac{g}{4}s_2+(g-1)s_3+(g-1)s_4+\frac{3}{2}(g-2)s_6+2(g-3)s_8.$$
	Since $g\geq 13$ and each  $s_i$ is a nonnegative integer, it follows that  $s_3=s_4=s_6=s_8=0$, and hence $11-\frac{g}{2}=\frac{g}{4}s_2$, i.e., $s_2+2=\frac{44}{g}$,
	which implies that $s_2=0$ and $g=22$.
	Substitute into the formulas of \autoref{key lemma of Xiao},
	one obtains that $K_S^2=K_f^2=10>9=9\chi(\mathcal{O}_S)$,
	which is a contradiction.
	
	(2) We show that if $s_8\neq 0$, then
	$$s_8=1,~s_6=s_4=s_3=0, ~s_2=2,~ g=14, K_S^2=8.$$
	In fact, by (1), we have already proven $s_{10}=0$.
	According to \eqref{eqn-3-2}, $s_8=1$ and
	$$7=\frac{g}{4}s_2+(g-1)s_3+(g-1)s_4+\frac{3}{2}(g-2)s_6.$$
    Since $g\geq 13$, it follows that $s_3=s_4=s_6=0$, and $gs_2=28$.
    Hence either $g=28,s_2=1$ or $g=14,s_2=2$.
    The first case is impossible; otherwise, by \eqref{eqn-3-3} one has $n=\frac{1}{2}$, which is a contradiction since $n$ is an integer. Finally, $K_S^2=8$ follows from the formula in \autoref{key lemma of Xiao}.
	
	(3) We show that $s_6=0$.
	Otherwise we would have $s_6=1$ again by \eqref{eqn-3-2}.
	Moreover, by the arguments (1) and (2) above, we have
	$s_{10}=s_8=0$. Moreover, we have $$\frac{g}{2}+4=(2g+1)-\frac{3}{2}(g-2)
	=\frac{g}{4}s_2+(g-1)s_3+(g-1)s_4.$$
	Hence $s_3=s_4=0$, $s_2=3$ and $g=16$.
	Substitute into \eqref{eqn-3-3}, one obtains that $n=\frac{1}{2}$, which gives a contradiction as $n$ is an integer.
	
	(4) We show that $s_3=s_4=0$.
	Otherwise by (2), $s_8=0$. Combining with (1) and (2) above,
	we have $s_6=s_8=s_{10}=0$, and  by \eqref{eqn-3-2}
	$$2g+1=\frac{g}{4}s_2+(g-1)(s_3+s_4).$$
	Hence $s_3+s_4\leq 2$.
	If $s_3+s_4=2$, then $gs_2=12$, which is impossible since $g\geq 13$.
	If $s_3+s_4=1$, then by \eqref{eqn-3-3} one gets $gn=4$, which is impossible since $g\geq 13$ and  $n\geq 1$.
	
	In conclusion, if $g\geq 13$, then except the possible case $(K_S^2,g)=(8,14)$, one has $s_i=0$ for all $i>2$.
	Hence by \eqref{eqn-3-3} it holds $gn=2$, which is impossible since $n$ is an integer.
	This completes the proof.
\end{proof}

	\begin{proposition}\label{prop-s_s>0-2}
	Let $S$ be a minimal   surface of general type with $p_g(S)=q(S)=1$.
		Suppose that the Albanese fibration $f:\,S\to B$ of $S$ is hyperelliptic of genus $g$.
	\begin{enumerate}
		\item The genus $g\neq 12$.
		\item If $g=11$, then $K_S^2=8$.
	\end{enumerate}
\end{proposition}
   \begin{proof}
   	(1) Suppose that $g=12$. Then by \eqref{eqn-3-1} together with \autoref{key lemma of Xiao},
   		$$25=2g+1=3s_2+11s_3+11s_4+15s_6+18s_8+20s_{10}+21s_{12}.$$
   	Note that all the $s_i$'s are non-negative integers.
   	It is not difficult to show that $s_6=s_8=s_{10}=s_{12}=0$.
   	Thus $3s_2+11(s_3+s_4)=25.$
   	It follows that $s_2=1$ and $s_3+s_4=2$.
   	By \autoref{key lemma of Xiao} one has
   	$6n=\frac{gn}{2}=1+s_3+s_4=3$, which is impossible since $n$ is an integer.
   	
   	(2) Suppose that $g=11$. Similar as above, by \eqref{eqn-3-1} together with \autoref{key lemma of Xiao},
   	$$46=2(2g+1)=\frac{11}{2}s_2+20s_3+20s_4+27s_6+32s_8+35s_{10}+36s_{12}.$$
   	In particular, $s_2$ is even.
   	If $s_2\geq 6$, then $$20s_3+20s_4+27s_6+32s_8+35s_{10}+36s_{12}\leq 13.$$
   	This is impossible since all $s_i$'s are non-negative integers.
   	If $s_2=4$, then
   	$$20s_3+20s_4+27s_6+32s_8+35s_{10}+36s_{12}=24.$$
   	It is again impossible.
   	If $s_2=2$, then
   	$$20s_3+20s_4+27s_6+32s_8+35s_{10}+36s_{12}=35.$$
   	Hence $s_{10}=1$ and $s_i=0$ for other $i>2$.
   	Combining with \autoref{key lemma of Xiao}, one computes $K_S^2=8$.
   	If $s_2=0$, then
   	$$20s_3+20s_4+27s_6+32s_8+35s_{10}+36s_{12}=46.$$
   	One checks again that this is impossible since all $s_i$'s are non-negative integers.
   	This completes the proof.
   \end{proof}
	
\begin{remark}
	Let $S$ be a minimal   surface of general type with $p_g(S)=q(S)=1$.
		Suppose that the Albanese fibration $f:\,S\to B$ of $S$ is hyperelliptic of genus $g$.
	By a similar argument as above, one can show that $g\neq 9$. More precisely, the list of all possibilities for $(K_S^2,g)$ is as follow:
$$\begin{tabular}{|c|c|} \hline
$~K_S^2~$ & $g$ \\\hline
$9$ & $4,~6,~8,~10$ \\\hline
$8$ & $3,~4,~5,~6,~7,~8,~10,~11,~14~$ \\\hline
$7$ & $3,~4,~5,~6$ \\\hline
$6$ & $2,~3,~4,~5,~6,~7,~8$ \\\hline
$5$ & $2,~3,~4$ \\\hline
$4$ & $2,~3,~4$ \\\hline
$~3,~2~$ & $2$ \\\hline
\end{tabular}$$
If $K_S^2=2$, then $g=2$ by the slope inequality \eqref{eqn-slope}.
If $K_S^2=3$, then one can only get $g\leq 4$ by the slope inequality.
Catanese-Ciliberto \cite[Prop\,5.6 \& Thm\,5.7]{cc91} proved that in this case $g=2$.
Our method here provides a different proof of this result.
Finally, it would be interesting to construct examples of hyperelliptic Albanese fibrations (or exclude the existence) in the above list $(K_S^2,g)$.
	\end{remark}	

%
%
%

	\section{Hyperelliptic Albanese fibrations with a quadratic Albanese genus}\label{sec-example}
	In this section, we will construct several examples.  \autoref{thm-example} and \autoref{thm-example-2} show  that the equalities in \autoref{pro-4-1} can be reached for both $g$ odd and $g$ even; in particular, the genus $g$ of hyperelliptic Albanese fibrations can be as large as a quadratic function in $\chi(\mathcal{O}_S)$.  \autoref{example-3}   indicates that the equality in \autoref{theorem upper bound for g} (1) is also sharp.
	\begin{example}\label{thm-example}
		There exist a sequence of minimal irregular surfaces $S_k$ ($k\geq 3$ and odd) of general type with $q(S_k)=1$,
		such that their Albanese maps induce a hyperelliptic fibration $f_k:\,S_k \to E$ of odd genus $g_k=\frac{(k-1)(k+3)}{2}+1$, and that
		$$\chi(\mathcal{O}_{S_k})=\chi_{f_k}=\frac{3k-1}{2},\qquad K_{S_k}^2=K_{f_k}^2=8k-8.$$
	\end{example}

	We use notations introduced in \autoref{sec-hyperelliptic}.
	To construct the required hyperelliptic fibrations is equivalent to find appropriate data $(P,R,\delta)$ with suitable singularity indices $s_i$'s.

	For any integer $m\geq 1$, let  $G_m=\mathbb{Z}/{m\mathbb{Z}}$ be the cyclic group of order $m$. Assume that   $\sigma\in G_m$ acts  on $\mathbb{P}^1$
	by $\sigma (t)=\xi t$ (where $\xi$ is any $m$-th primitive unit root), and $\sigma\in G_m$ acts  on some elliptic curve $E_0$ by translation $\sigma(x)=x+p$, where $p$ is a torsion point of order $m$. Let $P:=(E_0\times \mathbb{P}^1)/{G_m}$, where $G_m$ acts on $E_0\times \mathbb{P}^1$ by the diagonal action. Then we have the following commutative diagram
	
	$$\xymatrix{E_0\times \mathbb{P}^1 \ar[rr]^-{\Pi}\ar[d] && P:=(E_0\times \mathbb{P}^1)/{G_m} \ar[d]^h\\
		E_0\ar[rr]^{m:1} &&  E:=E_0/G_m}
	$$	
	Then $h:P\rightarrow E$ is a $\mathbb{P}^1$-bundle over the elliptic curve $E$. In fact, $P\cong \mathbb{P}(\mathcal{O}_{E}\oplus \mathcal{N})$, and it admits two sections $C_0,C_\infty$ with $C_0^2=C_\infty^2=0$, where $\mathcal{N}$ is a torsion line bundle of order $m$ over $E$.
	Denote by $\Gamma$ the general fiber of $h$.
	As a geometric ruled surface over an elliptic curve, the Picard group of $P$ is simply:
	$$\Pic(P)=\mathbb{Z}[C_0]\oplus h^*\Pic(E).$$
	
	\begin{lemma}\label{lem-5-1}
		Let $h:P\rightarrow E$ be the $\mathbb{P}^1$-bundle above.
		For any irreducible curve $D\subseteq P$, let $D \sim a C_0+b \Gamma$. Then
		$a\geq0$, $b\geq 0$, and $a+b>0$.
		Moreover, if $b=0$, then $a\geq m$ unless $D=C_0$ or $C_{\infty}$.
	\end{lemma}
	\begin{proof}
		We prove here that $a\geq m$ if $D\neq C_0$, $D\neq C_{\infty}$  and $b=0$. The rest is clear by \cite[\S\,V.2]{har77}.
		Let $D'=\Pi^{*}(D) \subseteq E_0\times \bbp^1$.
		Since $D\sim aC_0$, it follows that $D' \sim a C'$, where $C'$ stands for $E_0\times \{x\} \subseteq E_0\times \bbp^1$.
		In particular, $D'$ consists of several sections (the number is exactly $a$) of the trivial $\bbp^1$-bundle $E_0\times \bbp^1 \to E_0$.
		On the other hand, as $D\neq C_0$ and $D\neq C_{\infty}$,
		it follows that any irreducible component in $D'$ maps to another component under any non-identity element of the Galois group $G_m$.
		It follows that the number of irreducible components in $D'$ is a multiple of $m$,
		and in particular $a\geq m$ as required.		
	\end{proof}
%
	
	\begin{lemma}\label{lem-5-2}
		Let $h:P\rightarrow E$ be the $\mathbb{P}^1$-bundle above.
		For any $p\in P\setminus\{C_0\cup C_{\infty}\}$, let $\tau:\,\wt{P} \to P$ be the blow-up centered at $p$, and $\mathcal{E}$ be the exceptional curve.
		Then for any odd $3\leq k\leq m-2$, there exists a smooth divisor $\wt{R}\in \left|\tau^*\big((2g_k+2)C_0+2\Gamma\big)-2k\mathcal{E}\right|$, where $g_k=\frac{(k-1)(k+3)}{2}+1$.
	\end{lemma}
\begin{proof}
	It is enough to prove that the linear system $\left|\tau^*\big((2g_k+2)C_0+2\Gamma\big)-2k\mathcal{E}\right|$ is base-point-free.
	Note that $K_{\wt P}=\tau^{*}K_{P}+\mathcal{E}$.
	It follows that
	$\tau^*\big((2g_k+2)C_0+2\Gamma\big)-2k\mathcal{E}
	=K_{\wt P}+L$, where
	$$L=\tau^*\big((2g_k+2)C_0+2\Gamma-K_P\big)-(2k+1)\mathcal{E}.$$
	Note that numerically we have $K_P \sim -2C_0$. Hence $L^2=4k+11\geq 15>0$.
	We claim that
	\begin{claim}\label{claim-5-1}
		For any irreducible curve $\wt{D}\subseteq \wt P$, $L\cdot \wt{D} \geq 2$. In particular, $L$ is ample.
	\end{claim}
	Assuming the above claim, then one sees easily that the linear system
	$$\left|\tau^*\big((2g_k+2)C_0+2\Gamma\big)-2k\mathcal{E}\right|=\left|K_{\wt P}+L\right|$$
	is base-point-free by applying Reider's method \cite{reider88}.
	Hence a general divisor $\wt R \in \left|K_{\wt P}+L\right|$
	is smooth by Bertini's theorem.
	
	It remains to prove \autoref{claim-5-1}.
	Let $\wt{D}\subseteq \wt P$ be any irreducible curve.
	If $\wt{D}=\mathcal{E}$, or $\wt{D}$ is the strict transform of any fiber of $h:\,P\to E$,
	then one checks easily that
	$$L\cdot \wt{D}>2.$$
	Otherwise, $\wt{D}\sim \tau^*(aC_0+b\Gamma)-\beta\mathcal{E}$ with $a>0$,
	i.e., the restriction $h|_{\wt{D}}:\, \wt{D} \to E$ is surjective.
	Thus
	\begin{equation}\label{eqn-5-1}
		0\leq 2p_a(\wt{D})-2 =(K_{\wt P}+\wt{D})\cdot \wt{D},\qquad\text{namely,}\qquad
		2(a-1)b\geq \beta(\beta-1).
	\end{equation}
	By direct computation,
	$$L\cdot \wt{D} = 2\big(a+(g_k+2)b\big)-(2k+1)\beta=2a+(k+1)^2b-(2k+1)\beta+2b.$$

If $\beta\geq k+2$, then
	$$\begin{aligned}
	L\cdot \wt{D} &\,= 2+2(a-1)+(k+1)^2b-(2k+1)\beta+2b\\
	&\,\geq 2+2\sqrt{2(k+1)^2(a-1)b}-(2k+1)\beta\\
	&\,\geq 2+2(k+1)\sqrt{\beta(\beta-1)}-(2k+1)\beta\\
	&\,> 2;
	\end{aligned}$$
	if $\beta\leq k+1$, then
	$$\begin{aligned}
	L\cdot \wt{D} &\,= 2+2(a-1)+(k+1)^2b-(2k+1)\beta+2b\\
	&\,\geq \left\{\begin{aligned}
		&2+2(k+1)^2-(2k+1)(k+1)=k+3> 2, &\quad&\text{if~}b\geq 2;\\
		&2+\beta(\beta-1)+(k+1)^2-(2k+1)\beta=2+(k+1-\beta)^2> 2, && \text{if~}b=1,\\
		&2, && \text{if $b=0$ and $a=1$};\\
		&2a-(2k+1)\beta \geq 2m-(2k+1)>2, && \text{if $b=0$ and $a\geq 2$};
	\end{aligned}\right.
\end{aligned}$$

	Here we explain a little more about the case when $b=0$:
	if $a=1$, then $\wt{D}$ must be the strict transform of $C_0$ or $C_\infty$ by \autoref{lem-5-1},
	which implies that $\beta=0$ since $p\in P\setminus\{C_0\cup C_{\infty}\}$,
	and hence $L\cdot \wt{D}=2$;
	if $a\geq 2$, then $\beta\leq 1$ by \eqref{eqn-5-1}, and $a\geq m$ by \autoref{lem-5-1}.
\end{proof}

We come back to the construction of the required hyperelliptic Albanese fibrations in \autoref{thm-example}.
	For any odd $k\geq 3$, let $\wt{R}\in \left|\tau^*\big((2g_k+2)C_0+2\Gamma\big)-2k\mathcal{E}\right|$ be any smooth divisor on $\wt P$ as in \autoref{lem-5-2}.
	Let $R=\tau(\wt R)\subseteq P$ be its image in $P$.
	Then $R$ admits a singularity of order $2k$ at the point $p$, and
	$$\mathcal{O}_{P}(R) \cong  \delta^{\otimes 2},\qquad \text{where~}\delta=\mathcal{O}_{P}\big((g_k+1)C_0+\Gamma\big).$$
	Hence one can construct a hyperelliptic fibration $f_k:\, S_k \to E$ using the above data $(P,R,\delta)$.
	The genus $g_k$ of a general fiber of $f_k$ is $g_k=\frac{(k-1)(k+3)}{2}+1$.
	Moreover, using the notations introduced in \autoref{sec-hyperelliptic},
	one computes that $s_{2k}=1$, $s_2=(K_{\wt P/E}+\wt R)\cdot \wt R=10k$, and $s_i=0$ for other $i$.
	Hence by \autoref{key lemma of Xiao} (one can also compute $\chi(\mathcal{O}_{S_k})$ and $K_{S_k}^2$ by applying the formulas \cite[\S\,V.22]{bhpv} for double covers, as $\wt{R}$ is already smooth),
	$$\chi(\mathcal{O}_{S_k})=\chi_{f_k}=\frac{3k-1}{2},\qquad K_{S_k}^2=K_{f_k}^2=8k-8.$$
	Suppose that $q(S_k)=g(E)=1$. Then $f_k:\,S_k \to E$ is nothing but the Albanese fibration of $S_k$.
	Thus it remains to show that $q(S_k)=1$.
	After blowing up $\tau:\, \wt P \to P$,
	one sees that $S_k$ is a double cover of $\wt P$ branched over the smooth divisor $\wt{R}$.
	Moreover, similar to \autoref{claim-5-1}, one shows that $\wt{R}$ is ample.
	Hence by Kodaira's vanishing we get
 $$q(S_k)=h^1(\omega_{S_k})=h^1(\omega_{\wt P})+h^1(\omega_{\wt P}\otimes \tilde{\delta})=h^1(\omega_{\wt P})=g(E)=1$$ as required.

\begin{remark}\label{rem-5-3}
\autoref{thm-example} shows also that
	the upper bound in \autoref{pro-4-1} for $g$ odd is sharp.
	Indeed, the fibration  $f_k:\,S_k \to E$ is a relative minimal hyperelliptic fibration
 of odd genus $g_k=\frac{(k-1)(k+3)}{2}+1$, $q(S_k)=g(E)=1$,
 $n=\frac{R^2}{4(g+1)}=2$, and
 $$\left\{\begin{aligned}
 &\chi(\mathcal{O}_{S_k})=\chi_{f_k}=\frac{3k-1}{2},\\
 &K_{S_k}^2=K_{f_k}^2=8k-8.
 \end{aligned}\right.$$
 Hence
 $$\lambda_{f_k}=\frac{K_{S_k}^2}{\chi_{f_k}}= \frac{16}{3}-\frac{32}{3(3k-1)}.$$
 Therefore, one checks directly that
 $$\frac{(\lambda_{f_k}-4)^2}{4n}\chi_{f_k}^2+\left(\lambda_{f_k}-4+\frac{\lambda_{f_k}}{2n}\right)\chi_{f_k}+n+1=\frac{(k-1)(k+3)}{2}+1=g_k.$$
 \end{remark}
%

Using a similar method, we can also construct examples reaching the equality in \autoref{pro-4-1} for  $g$ even.
\begin{example}\label{thm-example-2}
		There exist a sequence of minimal irregular surfaces $S_k$ ($k\geq 1$ and odd) of general type with $q(S_k)=1$,
		such that their Albanese maps induce a hyperelliptic fibration $f_k:\,S_k \to E$ of even genus $g_k=(k+1)^2$, and that
		$$\chi(\mathcal{O}_{S_k})=\chi_{f_k}=\frac{3k+1}{2},\qquad K_{S_k}^2=K_{f_k}^2=8k-2, \quad n=1$$
and the equality in \autoref{thm-sharp-upper-bound} and  \autoref{pro-4-1} holds.

	\end{example}

  The proof is similar to \autoref{thm-example}, the difference is that here we take $P=\mathbb{P}_E(V)$, where $V$ is an indecomposable rank 2 vector bundle over an elliptic curve $E$ with $\deg V =1$.

  We use notations as in the construction of \autoref{thm-example}.
  Let  $h:P\rightarrow E$ be the $\mathbb{P}^1$-bundle over the elliptic curve $E$. We denote by
 $C$  a section of $h$ with $C^2=1$ and by $\Gamma$ a general fibre of $h$. Then we have
 $$ K_{P/E}=K_P\sim -2C+\Gamma,\quad C\cdot \Gamma =1, \quad (K_P)^2=0, \quad K_P\cdot C=-1, \quad K_P\cdot \Gamma =-2.$$

  The key point is the following

 \begin{lemma}\label{lemma-5-2}
 For any $p\in P$, let $\tau:\,\wt{P} \to P$ be the blow-up centered at $p$, and $\mathcal{E}$ be the exceptional curve.
		Then for any odd $k\geq 1$ and $g_k=(k+1)^2$, there exists a smooth divisor
$\wt{R}\in    \left|\tau^*\big(-(g_k+1)K_{P/E}+\Gamma\big)-2k\mathcal{E}\right|=\left|\tau^*\big((2g_k+2)C-g\Gamma\big)-2k\mathcal{E}\right|$.
 \end{lemma}

  \begin{proof}
	It is enough to prove that the linear system $\left|\tau^*\big((2g_k+2)C-g\Gamma\big)-2k\mathcal{E}\right|$ is base-point-free.
	Note that $K_{\wt P}=\tau^{*}K_{P}+\mathcal{E}$.
	It follows that
	$\tau^*\big((2g_k+2)C-g\Gamma\big)-2k\mathcal{E}
	=K_{\wt P}+L$, where
	$$L=\tau^*\left((2g_k+4)C-(g+1)\Gamma\right)-(2k+1)\mathcal{E}, \quad L^2=4k+11\geq 15.$$
	
	We claim that
	\begin{claim}\label{claim-5-2}
		For any irreducible curve $\wt{D}\subseteq \wt P$, $L\cdot \wt{D} \geq 2$. In particular, $L$ is ample.
	\end{claim}
	Assuming the above claim, then one sees easily that the linear system
	$$\left|\tau^*\big((2g_k+2)C-g\Gamma\big)-2k\mathcal{E}\right|=\left|K_{\wt P}+L\right|$$
	is base-point-free by applying Reider's method \cite{reider88}.
	Hence a general divisor $\wt R \in \left|K_{\wt P}+L\right|$
	is smooth by Bertini's theorem.
	
	It remains to prove \autoref{claim-5-2}.
	Let $\wt{D}\subseteq \wt P$ be any irreducible curve.
	If $\wt{D}=\mathcal{E}$, or $\wt{D}$ is the strict transform of any fiber of $h:\,P\to E$,
	then one checks easily that
	$$L\cdot \wt{D}>2.$$
	Otherwise, assume $\wt{D}\sim \tau^*(aC+b\Gamma)-\beta\mathcal{E}$ with $a>0$, then we have $a+2b\geq 0$ since $\wt{D}$ is an irreducible curve.
	Since the restriction $h|_{\wt{D}}:\, \wt{D} \to E$ is surjective,  we have
	\begin{equation}\label{eqn-5-11}
		0\leq 2p_a(\wt{D})-2 =(K_{\wt P}+\wt{D})\cdot \wt{D},\qquad\text{namely,}\qquad
		(a+2b)(a-1)\geq \beta(\beta-1).
	\end{equation}
	By direct computation,
	$$L\cdot \wt{D} = (a+2b)g_k +3a+4b-(2k+1)\beta=(a+2b)(k+1)^2+a-(2k+1)\beta+2(a+2b).$$

 (i) If $a+2b=0$, then  we have $b\leq -1$  and $\wt{D}\sim -b(\tau^*K_P)$. Hence we get $L\cdot \wt{D}=-2b\geq 2$;

 (ii)  if $a+2b=1$, then  we have $a\geq \beta(\beta-1)+1$ by \eqref{eqn-5-11}. Hence we get
$$L\cdot \wt{D} \geq (k+1)^2+\beta(\beta-1)+1-(2k+1)\beta+2=(k+1-\beta)^2+3> 2;$$

(iii) if $a+2b\geq 2$, we discuss the following two subcases separately:
	if $\beta\geq k+2$, then
	$$\begin{aligned}
	L\cdot \wt{D} &\,= (a+2b)(k+1)^2+(a-1)-(2k+1)\beta+2(a+2b)+1\\
	&\,\geq 2\sqrt{(k+1)^2(a+2b)(a-1)}-(2k+1)\beta+2(a+2b)+1\\
	&\,\geq 2(k+1)\sqrt{\beta(\beta-1)}-(2k+1)\beta+2(a+2b)+1\\
	&\,> 2(a+2b)+1\geq 1,
	\end{aligned}$$
Since $L\cdot \wt{D}$ is an integer, we get $L\cdot \wt{D}\geq 2$;
	if $\beta\leq k+1$, then
\[	L\cdot \wt{D} \geq 2(k+1)^2-(2k+1)(k+1)+4=k+5> 2. \qedhere\]
\end{proof}

We come back to the construction of \autoref{thm-example-2}.
	For any odd $k\geq 1$, let $$\wt{R}\in \left|\tau^*\big((2g_k+2)C-g\Gamma\big)-2k\mathcal{E}\right|$$ be any smooth divisor on $\wt P$ as in \autoref{lem-5-2}.
	Let $R=\tau(\wt R)\subseteq P$ be its image in $P$.
	Then $R$ admits a singularity of order $2k$ at the point $p$, and
	$$\mathcal{O}_{P}(R) \cong  \delta^{\otimes 2},\qquad \text{where~}\delta=\mathcal{O}_{P}\big((g_k+1)C-g\Gamma\big).$$
	Hence one can construct a hyperelliptic fibration $f_k:\, S_k \to E$ using the above data $(P,R,\delta)$.
	The genus $g_k$ of a general fiber of $f_k$ is $g_k=(k+1)^2$.
	Moreover, using the notations introduced in \autoref{sec-hyperelliptic},
	one computes that $s_{2k}=1$, $s_2=(K_{\wt P/E}+\wt R)\cdot \wt R=10k+6$, and $s_i=0$ for other $i$.
	Hence by \autoref{key lemma of Xiao} (one can also compute $\chi(\mathcal{O}_{S_k})$ and $K_{S_k}^2$ by applying the formulas \cite[\S\,V.22]{bhpv} for double covers, as $\wt{R}$ is already smooth),
	$$\chi(\mathcal{O}_{S_k})=\chi_{f_k}=\frac{3k+1}{2},\qquad K_{S_k}^2=K_{f_k}^2=8k-2.$$
	Suppose that $q(S_k)=g(E)=1$. Then $f_k:\,S_k \to E$ is nothing but the Albanese fibration of $S_k$.
	Thus it remains to show that $q(S_k)=1$.
	After blowing up $\tau:\, \wt P \to P$,
	one sees that $S_k$ is a double cover of $\wt P$ branched over the smooth divisor $\wt{R}$.
	Moreover, similar to \autoref{claim-5-2}, one shows that $\wt{R}$ is ample.
	Hence, arguing as above $q(S_k)=q(\wt P)=g(E)=1$ as required.
Finally one checks directly that $n=1$, $\lambda_{f_k}=\frac{16}{3}-\frac{28}{3(3k+1)}$ and
 $$\frac{(\lambda_{f_k}-4)^2}{4n}\chi_{f_k}^2+\left(\lambda_{f_k}-4+\frac{\lambda_{f_k}}{2n}\right)\chi_{f_k}+n+1=(k+1)^2=g_k.$$
This completes the construction of \autoref{thm-example-2}.\vspace{2mm}

At the end of this section, we construct examples showing that the inequality in \autoref{theorem upper bound for g} (1) is also sharp.
\begin{example}\label{example-3}
  In \autoref{thm-example-2},   for any integer $n\geq 1$, $\chi_n\geq 6$ such that $n|(2\chi_n+2)$, take $g_n:=\frac{2\chi_n+2}{n}\geq 2$, $g_n\equiv n-1 (\mod 2) $ and   $k=2$,
  $$\wt{R}\in \left|\tau^*\big(-(g_n+1)K_{P/E}+n\Gamma\big)-4\mathcal{E}\right| = \left|\tau^*\big((2g_n+2)C+(n-1-g_n)\Gamma\big)-4\mathcal{E}\right|,$$
  we can  get a sequence of minimal irregular surfaces $S_n$ ($n\geq 1$) of general type with $q(S_n)=1$,
		such that their Albanese maps induce a hyperelliptic fibration $f_n:\,S_n \to E$ of genus $g_n=\frac{2\chi_n+2}{n}$, and that
		$$\chi(\mathcal{O}_{S_n})=\chi_{f_n}=\chi_n,\qquad K_{S_n}^2=K_{f_n}^2=4\chi_{f_n}-2(n-1)\leq 4\chi_{f_n}, \quad g_n=\frac{4\chi_{f_n}+4}{2+4\chi_{f_n}-K_{f_n}^2}.$$

The construction is the same as in \autoref{thm-example-2} and we only need to prove the following
\begin{claim}
 For $L=\tilde{R}-K_{\wt{P}}=\tau^*\left((2g_n+4)C+(n-2-g)\Gamma\right)-5\mathcal{E}$,
we have $L^2\geq 5$ and for  any irreducible curve
 $\wt{D}\subseteq \wt P$, $L\cdot \wt{D} \geq 2$.
\end{claim}
\begin{proof}
 Here we have  $n\geq 1, ng_n=2\chi_n+2\geq 14$ and
 $$L^2=(2g_n+4)^2+(2g_n+4)(n-g_n-1)-25=4ng_n+8n+2(2g_n+4)-25>37>5.$$
Now let $\wt{D}\subseteq \wt P$ be any irreducible curve.
	If $\wt{D}=\mathcal{E}$, or $\wt{D}$ is the strict transform of any fiber of $h:\,P\to E$,
	then one checks easily that$L\cdot \wt{D}>2$.
	Otherwise, assume $\wt{D}\sim \tau^*(aC+b\Gamma)-\beta\mathcal{E}$ with $a>0$, then we have $a+2b\geq 0$ and $(a+2b)(a-1)\geq \beta(\beta-1)$ as in \autoref{lemma-5-2}.
By direct computation, we get
	$$L\cdot \wt{D} = (a+2b)g_n +an+2(a+2b)-5\beta.$$
 If $a+2b=0$ or $1$, one can show $L\cdot \wt{D} \geq 2$ as in \autoref{lemma-5-2}.
Now assume $a+2b\geq 2$, if  $\beta\leq 1$, we have $L\cdot \wt{D}\geq 2g_n+an+4-5\geq 4$ since
$g_n\geq 2$; if $\beta\geq 2$, we have $2\beta(\beta-1)\geq \beta^2$ and thus
$$\begin{aligned}
	L\cdot \wt{D} &\,= (a+2b)g_n+(a-1)n+n+4-5\beta\\
	&\,\geq \sqrt{4ng_n(a+2b)(a-1)}-5\beta+n+4\\
	&\,\geq \sqrt{56\beta(\beta-1)}-5\beta+n+4\\
	&\,>n+4\geq5.
	\end{aligned}$$
\end{proof}
\end{example}

\end{document}